\documentclass[reqno,12pt]{amsart}

\usepackage[utf8]{inputenc}
\usepackage{microtype}
\usepackage{amsfonts}
\usepackage{amsmath}
\usepackage{graphicx}
\usepackage{float}
\usepackage[dvipsnames]{xcolor}
\usepackage{hyperref}
\usepackage{tikz-cd}

\usepackage{amssymb}
\usepackage{enumitem}
\usepackage{mathrsfs}

\usepackage{tikz}
\usetikzlibrary{topaths}
\usetikzlibrary{calc}

\usepackage{a4wide}

\usepackage{empheq}

\newtheorem{theorem}{Theorem}[section]
\newtheorem*{theorem*}{Theorem}
\newtheorem*{maintheorem*}{Main Theorem}
\newtheorem{lemma}[theorem]{Lemma}
\newtheorem{proposition}[theorem]{Proposition}
\newtheorem*{proposition*}{Proposition}
\newtheorem{corollary}[theorem]{Corollary}
\newtheorem*{corollary*}{Corollary}

\newtheorem{cit}[theorem]{Citation}
\newtheorem*{conjecture*}{Conjecture}

\newtheorem*{question*}{Question}

\theoremstyle{definition}
\newtheorem{definition}[theorem]{Definition}
\newtheorem*{definition*}{Definition}

\newtheorem{example}[theorem]{Example}

\newcommand{\N}{\mathbb{N}}
\newcommand{\Z}{\mathbb{Z}}
\newcommand{\Q}{\mathbb{Q}}
\newcommand{\R}{\mathbb{R}}

\newcommand{\Cantor}{\mathfrak{C}}
\newcommand{\tree}{\mathcal{T}}

\DeclareMathOperator{\Aut}{Aut}

\DeclareMathOperator{\Homeo}{Homeo}

\DeclareMathOperator{\GL}{GL}
\DeclareMathOperator{\Aff}{Aff}

\DeclareMathOperator{\sgn}{sgn}

\DeclareMathOperator{\F}{F}

\DeclareMathOperator{\id}{id}

\numberwithin{equation}{section}

\begin{document}

\title[Embedding f.p. self-similar groups into f.p. simple groups]{Embedding finitely presented self-similar groups\\ into finitely presented simple groups}
\date{\today}


\author[M.~C.~B.~Zaremsky]{Matthew C.~B.~Zaremsky}
\address{Department of Mathematics and Statistics, University at Albany (SUNY), Albany, NY}
\email{mzaremsky@albany.edu}

\begin{abstract}
We prove that every finitely presented self-similar group embeds in a finitely presented simple group. This establishes that every group embedding in a finitely presented self-similar group satisfies the Boone--Higman conjecture. The simple groups in question are certain commutator subgroups of R\"over--Nekrashevych groups, and the difficulty lies in the fact that even if a R\"over--Nekrashevych group is finitely presented, its commutator subgroup might not be. We also discuss a general example involving matrix groups over certain rings, which in particular establishes that every finitely generated subgroup of $\GL_n(\Q)$ satisfies the Boone--Higman conjecture.
\end{abstract}

\maketitle
\thispagestyle{empty}

\section{Introduction}

The Boone--Higman conjecture, posed in the early 1970's by Boone and Higman \cite{boone74}, predicts that every finitely generated group with solvable word problem embeds in a finitely presented simple group. Recall that a group is \emph{simple} if it has no proper non-trivial quotients, and a finitely generated group has \emph{solvable word problem} if there exists an algorithm to determine whether or not any given word in the generators represents the identity. Prominent examples of groups known to satisfy the conjecture include all hyperbolic groups \cite{bbmz_hyp}, and every $\GL_n(\Z)$ \cite{scott84}. Since satisfying the conjecture is inherited by subgroups, this also turns out to establish the conjecture for prominent groups like virtually nilpotent groups and virtually special groups. For a more thorough summary of the state of art on this conjecture, see the survey \cite{bbmz_survey}.

One reason that a full resolution of the conjecture remains elusive is a lack of examples of finitely presented infinite simple groups. If one is striving to prove the conjecture for some given group, this is an immediate impediment. Thus, it is desirable to find sufficient conditions to embed in a finitely presented simple group, that do not actually require one to deal directly with simple groups. The main result of this paper is such a sufficient condition:

\begin{theorem}\label{thrm:main}
Every finitely presented self-similar group embeds in a finitely presented simple group, and hence satisfies the Boone--Higman conjecture.
\end{theorem}

See Definition~\ref{def:ss} for the definition of self-similar group. Intuitively, a self-similar group is a group of automorphisms of a regular rooted tree that resembles itself when restricted to certain subtrees. Finding a finitely presented self-similar group in which to embed is, in theory, an easier problem than finding a finitely presented simple group in which to embed, at least for groups that stand a chance of embedding in a self-similar group (e.g., such a group is necessarily residually finite). The simple groups used in Theorem~\ref{thrm:main} are certain commutator subgroups of R\"over--Nekrashevych groups, which are always simple, but in general are not always finitely presented, so some care in needed. We should mention that the proof of Theorem~\ref{thrm:main} also shows that every self-similar group of type $\F_n$ embeds in a simple group of type $\F_n$, though we will not focus on this here.

As an example, we establish a general criterion in Example~\ref{ex:ring} for $\GL_n(R)$ to embed in a finitely presented self-similar group, for $R$ a ring with certain properties. Thanks to Theorem~\ref{thrm:main}, this criterion does not require any knowledge about the wreath recursions involved, how they interact in the abelianization, or which self-similar representations lead to nice R\"over--Nekrashevych groups. We use this framework to prove the following:

\begin{theorem}\label{thrm:Q}
Every finitely generated subgroup of $\GL_n(\Q)$ satisfies the Boone--Higman conjecture.
\end{theorem}

This is striking since it is an open problem to find an explicit finitely presented group (even a non-simple one) into which $\GL_n(\Q)$ embeds. See also \cite[Problem~5.3(5)]{bbmz_survey}. Also, we should mention that perhaps the most prominent family of finitely generated subgroups of $\GL_n(\Q)$ is $S$-arithmetic groups (in characteristic 0) \cite{borel76}, so now we know they all satisfy the Boone--Higman conjecture.

\medskip

As a remark, it is not difficult to adapt our proof for finitely presented self-similar groups to also handle finitely generated self-similar groups that are ``contracting'' as in \cite{nekrashevych05}, even if they are not finitely presented. However, this is not new: by \cite{bbmz_hyp}, all contracting self-similar groups satisfy the Boone--Higman conjecture (not even requiring finite generation). In fact, the result in \cite{bbmz_hyp} implies that every R\"over--Nekrashevych group of a contracting self-similar group satisfies the conjecture, which is stronger still. The result in \cite{bbmz_hyp} is also much more general, dealing with so called contracting rational similarity groups; in \cite{belkmatuccicontracting}, Belk and Mutucci spell out the proof just for contracting self-similar and contracting R\"over--Nekrashevych groups.

\medskip

This paper is organized as follows. In Section~\ref{sec:ssrn} we recall some background on self-similar groups and R\"over--Nekrashevych groups, and prove a result about embedding wreath products (Proposition~\ref{prop:wreath}). In Section~\ref{sec:proof_main} we prove our main result, Theorem~\ref{thrm:main}. In Section~\ref{sec:example} we establish our general example involving matrix groups over certain rings, and prove Theorem~\ref{thrm:Q}.

\subsection*{Acknowledgments} Some of the ideas in Subsection~\ref{ssec:examples} arose from conversations at the conference \emph{Topological and Homological Methods in Group Theory 2024}, in Bielefeld, Germany in March 2024, in particular with Henry Bradford, Francesco Fournier-Facio, Eduard Schesler, Stefan Witzel, Xiaolei Wu, and Julian Wykowski. I am also grateful to Jim Belk, Francesco Matucci, Mark Pengitore, and Rachel Skipper for helpful discussions, and the referees for some excellent suggestions. The author is supported by grant \#635763 from the Simons Foundation.

\section{Self-similar groups and R\"over--Nekrashevych groups}\label{sec:ssrn}

Let $d\ge 2$ and let $\tree_d$ be the rooted $d$-regular tree. This is the tree whose vertex set is $\{1,\dots,d\}^*$, the set of finite words in the alphabet $\{1,\dots,d\}$, with an edge from $w$ to $wi$ for each vertex $w$ and each $i\in\{1,\dots,d\}$. The empty word $\varnothing$ is the \emph{root} of the tree. For each $w$, the vertices $w1,\dots,wd$ are the \emph{children} of $w$, and any vertex of the form $ww'$ for $w'\in\{1,\dots,d\}^*$ is a \emph{descendent} of $w$. If two vertices are not descendents of each other, call them \emph{incomparable}.

An \emph{automorphism} of $\tree_d$ is a self-bijection of $\{1,\dots,d\}^*$ that preserves adjacency. Denote by $\Aut(\tree_d)$ the group of all automorphisms of $\tree_d$. Since the root is the only vertex of degree $d$ (the others all have degree $d+1$), every automorphism fixes it. Thus, the set of children of the root is stabilized by every automorphism. This gives us a map $\Aut(\tree_d) \to S_d$ onto the symmetric group, which splits via the action of $S_d$ on $\tree_d$ that just permutes the first letter of a word. The kernel of this map is the pointwise stabilizer of the set of children of the root, which is isomorphic to $\Aut(\tree_d)^d$. Moreover, $S_d$ acts by conjugation on $\Aut(\tree_d)^d$ by permuting the coordinates. All in all, we have an isomorphism
\[
\Aut(\tree_d) \to S_d \wr_d \Aut(\tree_d)\text{.}
\]
Here we write the acting group on the left in the wreath product, for future convenience. The notation $\wr_d$ indicates that this is the wreath product coming from the action of $S_d$ specifically on $\{1,\dots,d\}$.

\begin{definition}[Self-similar]\label{def:ss}
A subgroup $G\le \Aut(\tree_d)$ is \emph{self-similar} if its image in $S_d\wr_d\Aut(\tree_d)$ under the above isomorphism is contained in $S_d\wr_d G$.
\end{definition}

See \cite{nekrashevych05} for more background on self-similar groups.

Given $g\in\Aut(\tree_d)$, if the image of $g$ under the isomorphism $\Aut(\tree_d)\to S_d \wr_d \Aut(\tree_d)$ is $\sigma(g_1,\dots,g_d)$, then we write
\[
g \leftrightarrow \sigma(g_1,\dots,g_d)
\]
and refer to this as the \emph{wreath recursion} of $g$. The $g_i$ are called the \emph{level-1 states} of $g$.

\medskip

The boundary at infinity of $\tree_d$ is the usual $d$-ary Cantor set $\Cantor_d\coloneqq \{1,\dots,d\}^\N$. Since an automorphism of $\tree_d$ fixing the boundary pointwise must be trivial, we have an embedding $\Aut(\tree_d)\to \Homeo(\Cantor_d)$. For each $w\in\{1,\dots,d\}^*$, define the \emph{cone} on $w$ to be $\Cantor_d(w)=\{w\omega\mid \omega\in\Cantor_d\}$, and let
\[
h_w \colon \Cantor_d \to \Cantor_d(w)
\]
be the \emph{canonical homeomorphism} sending $\omega$ to $w\omega$.

Now we describe a construction that takes as input a self-similar group $G\le\Aut(\tree_d)$ and outputs a group $V_d(G)$ of homeomorphisms of $\Cantor_d$ that has simple commutator subgroup, thus entering the world of simple groups.

\begin{definition}[R\"over--Nekrashevych group]\label{def:rn}
Let $G\le \Aut(\tree_d)$ be self-similar. The \emph{R\"over--Nekrashevych group} $V_d(G)$ is the group of homeomorphisms of $\Cantor_d$ of the following form:
\begin{itemize}
    \item Partition $\Cantor_d$ into some number $n$ of cones $\Cantor_d(w_1),\dots,\Cantor_d(w_n)$.
    \item Partition $\Cantor_d$ again, into the same number of cones $\Cantor_d(w_1'),\dots,\Cantor_d(w_n')$.
    \item Now for each $1\le i\le n$, send $\Cantor_d(w_i)$ homeomorphically to $\Cantor_d(w_i')$ via $h_{w_i'} \circ g_i \circ h_{w_i}^{-1}$ for some $g_i\in G$.
\end{itemize}
\end{definition}

When $G=\{1\}$ we get the classical Higman--Thompson group $V_d=V_d(\{1\})$, as in \cite{higman74}. This general construction of $V_d(G)$ was first envisioned by Scott in \cite{scott84general}, in much different language. Subsequently, in \cite{roever99} R\"over analyzed the case when $G$ is the Grigorchuk group from \cite{grigorchuk80}, and in \cite{nekrashevych04} Nekrashevych solidified the groups in full generality and with modern terminology. One notable application of R\"over--Nekrashevych groups is that they lead to the first examples of simple groups with arbitrary finiteness length \cite{skipper19}.

\medskip

The following two results make clear why R\"over--Nekrashevych groups are of interest to the Boone--Higman conjecture.

\begin{cit}\cite[Theorem~2]{scott84general}\label{cit:fp}
For any self-similar $G\le\Aut(\tree_d)$, if $G$ is finitely presented then so is $V_d(G)$.
\end{cit}

\begin{cit}\cite[Theorem~4.7]{nekrashevych18}\label{cit:simple}
For any self-similar $G\le\Aut(\tree_d)$, the commutator subgroup $[V_d(G),V_d(G)]$ is simple.
\end{cit}

The reason that this does not immediately help with the conjecture is that even if $V_d(G)$ is finitely presented, its commutator subgroup might not be. Of course if the abelianization of $V_d(G)$ is finite, then the commutator subgroup has finite index and is thus finitely presented. Our goal in the next section is therefore to cleverly avoid the infinite abelianization case.

First we record the following result about wreath products. Recall our convention that in wreath products the acting group is on the left. Also, we will write $\wr$ (with no subscript) to indicate that it is the wreath product coming from the acting group acting on itself by left translation.

\begin{proposition}\label{prop:wreath}
Let $G\le \Aut(\mathcal{T}_d)$ be a self-similar group of automorphisms of the $d$-ary tree $\mathcal{T}_d$. Let $V_d(G)$ be the corresponding R\"over--Nekrashevych group. For any finite group $F$, the wreath product $F\wr [V_d(G),V_d(G)]$ embeds into $[V_d(G),V_d(G)]$.
\end{proposition}

\begin{proof}
Let $n=|F|$. For each $1\le i\le n$, let $w_i=1^i 2 \in \{1,\dots,d\}^*$. Note that $w_1,\dots,w_n$ are pairwise incomparable vertices of $\tree_d$. For each $i$, consider the canonical homeomorphism $h_{w_i}\colon \Cantor_d \to \Cantor_d(w_i)$. This conjugates $\Homeo(\Cantor_d)$ isomorphically to $\Homeo(\Cantor_d(w_i))$, which can be viewed as a subgroup of $\Homeo(\Cantor_d)$ by extending every homeomorphism of $\Cantor_d(w_i)$ to a homeomorphism of $\Cantor_d$ by declaring that it acts trivially on $\Cantor_d\setminus\Cantor_d(w_i)$. In particular, for $i\ne j$ the subgroups $\Homeo(\Cantor_d(w_i))$ and $\Homeo(\Cantor_d(w_j))$ of $\Homeo(\Cantor_d)$ commute. From the definition of $V_d(G)$, it is clear that the image of $V_d(G)$ in $\Homeo(\Cantor_d(w_i))$ under this conjugation lands in $V_d(G)$, and hence the image of $[V_d(G),V_d(G)]$ lands in $[V_d(G),V_d(G)]$. By now we have embedded $([V_d(G),V_d(G)])^n$ into $[V_d(G),V_d(G)]$. Now fix a bijection $F\leftrightarrow\{w_1,\dots,w_n\}$, and use the action of $F$ on itself by translation to induce an action of $F$ on $\Cantor_d$ by permuting the cones $\Cantor_d(w_i)$ via the ``prefix replacement'' homeomorphisms $h_{w_j}\circ h_{w_i}^{-1}$ (and acting trivially outside these cones). This embeds $F$ into $V_d(G)$, in fact into $V_d$. Up to choosing a pair of disjoint cones that are also disjoint from all the cones $\Cantor_d(w_i)$, and extending this action of $F$ so that any elements acting by odd permutations also transpose those cones via prefix replacements, we can assume all elements of $F$ act as even permutations. By \cite{higman74} this means our embedding of $F$ lands in $[V_d,V_d]$, hence in  $[V_d(G),V_d(G)]$. Our embedded copy of $([V_d(G),V_d(G)])^n$ is normalized by our embedded copy of $F$, and the conjugation action just permutes the entries, and so we have embedded $F\wr [V_d(G),V_d(G)]$ into $[V_d(G),V_d(G)]$.
\end{proof}

\section{Proof of the main result}\label{sec:proof_main}

We now prove our main result, Theorem~\ref{thrm:main}, that every finitely presented self-similar group embeds in a finitely presented simple group.

\begin{proof}[Proof of Theorem~\ref{thrm:main}]
Let $G\le \Aut(\mathcal{T}_d)$ be a finitely presented self-similar group of automorphisms of the $d$-ary tree $\mathcal{T}_d$. Let $V_d(G)$ be the corresponding R\"over--Nekrashevych group. By Citation~\ref{cit:simple}, the commutator subgroup $[V_d(G),V_d(G)]$ is simple. By Citation~\ref{cit:fp}, since $G$ is finitely presented so is $V_d(G)$. Thus, if $V_d(G)$ has finite abelianization, its commutator subgroup is a finitely presented simple group. If the abelianization is infinite though, there is no reason \emph{a priori} to expect $[V_d(G),V_d(G)]$ to be finitely presented. Our strategy is to prove that there exists a possibly different embedding of $G$ into $\Aut(\mathcal{T}_{d'})$ for a possibly different $d'$, such that this copy of $G$ is still self-similar, but now the R\"over--Nekrashevych group $V_{d'}(G)$ for this self-similar representation of $G$ has finite abelianization.

The abelianization of any R\"over--Nekrashevych group $V_d(G)$ is given by a straightforward procedure, described for example in \cite[Theorem~4.8]{nekrashevych18}. For each $g\in G$, consider the wreath recursion $g\leftrightarrow \sigma(g_1,\dots,g_d)$. Let $\overline{G}$ be the abelianization of $G$, and write $\overline{g}$ for the image of $g\in G$ in $\overline{G}$. For $d$ even, the abelianization of $V_d(G)$ is the result of taking $\overline{G}$ and modding out the relation $\overline{g} = \overline{g}_1+\cdots+\overline{g}_d$ coming from the above wreath recursion, for every $g\in G$. For $d$ odd, the abelianization is the result of taking $\overline{G}\oplus (\Z/2\Z)$ and modding out the relation $\overline{g} = \overline{g}_1+\cdots+\overline{g}_d+\sgn(\sigma)$ coming from the above wreath recursion, for every $g\in G$. Here $\sgn(\sigma)\in\Z/2\Z$ equals $0+2\Z$ if $\sigma$ is an even permutation and $1+2\Z$ if $\sigma$ is an odd permutation.

Given $G\le \Aut(\mathcal{T}_d)$ as above, and given $m\in\mathbb{N}$, consider the action of $G$ on $\mathcal{T}_{md}$ defined by the wreath recursion
\[
g \leftrightarrow \sigma^{\oplus m} (g_1,\dots,g_d,g_1,\dots,g_d,\dots,g_1,\dots,g_d) \text{,}
\]
where $\sigma^{\oplus m}$ is the element of $S_{md}$ that acts like $\sigma$ on $\{1,\dots,d\}$, acts like $\sigma$ in the obvious way on $\{d+1,\dots,2d\}$, and so forth. Since the original action of $G$ is self-similar, each $g_i$ lies in $G$, and so this new action is also self-similar. Since the original action is faithful, it is easy to see that this new action is also faithful. Thus, we have embedded $G$ into $\Aut(\mathcal{T}_{md})$ as a self-similar group with this wreath recursion. Note that if $m$ is even then the abelianization of $V_{md}(G)$ is given by taking the abelianization of $G$ and then modding out the relation $\overline{g} = m(\overline{g}_1+\cdots+\overline{g}_d)$ for every $g\in G$. For convenience we will just focus on even $m$ from now on.

We now claim that there exists some even $m$ such that this abelianization is finite. First note that it suffices to confirm finiteness after modding out just those equations coming from every generator of $G$; say $G$ is generated by $\{a_1,\dots,a_n\}$. In other words, we want to prove that if we take the free abelian group generated by $\overline{a}_1,\dots,\overline{a}_n$, and mod out the relation $\overline{a}_i = m\big((\overline{a}_i)_1+\cdots+(\overline{a}_i)_d\big)$ for every $1\le i\le n$, we get a finite abelian group for some even $m$. Equivalently, we can take $\R^n$ with standard basis $\{\overline{a}_1,\dots,\overline{a}_n\}$, mod out the span of all the $\overline{a}_i - m\big((\overline{a}_i)_1+\cdots+(\overline{a}_i)_d\big)$, and prove we get 0 for some even $m$. Set $A$ equal to the $n$-by-$n$ matrix whose $i$th column is $(\overline{a}_i)_1+\cdots+(\overline{a}_i)_d$, and choose some even number $m\in\N$ such that $1/m$ is not an eigenvalue of $A$. Then $1$ is not an eigenvalue of $mA$, so the columns of $I_n-mA$ span all of $\R^n$. This confirms that modding out the span of all the $\overline{a}_i - m((\overline{a}_i)_1+\cdots+(\overline{a}_i)_d)$ yields the trivial vector space.

By now we have proved that our finitely presented group $G$ admits a faithful self-similar representation into some $\Aut(\mathcal{T}_{d'})$ such that $V_{d'}(G)$ has finite abelianization. Thus, the commutator subgroup $[V_{d'}(G),V_{d'}(G)]$ has finite index in $V_{d'}(G)$, hence is finitely presented by Citation~\ref{cit:fp}. It is also simple by Citation~\ref{cit:simple}, so now it suffices to prove that $G$ embeds in $[V_{d'}(G),V_{d'}(G)]$. Since $G$ embeds in $V_{d'}(G)$, some finite index subgroup $H\le G$ embeds in $[V_{d'}(G),V_{d'}(G)]$. Without loss of generality $H$ is normal in $G$, so by \cite{KK} $G$ embeds in the wreath product $(G/H) \wr H$. Since $F\wr [V_{d'}(G),V_{d'}(G)]$ embeds in $[V_{d'}(G),V_{d'}(G)]$ for any finite group $F$ by Proposition~\ref{prop:wreath}, we finish with embeddings
\[
G \hookrightarrow (G/H)\wr H \hookrightarrow [V_{d'}(G),V_{d'}(G)] \text{.}
\]
\end{proof}

\section{Examples from virtual endomorphisms}\label{sec:example}

One convenient way to produce examples of self-similar groups is via virtual endomorphisms; see \cite{nekrashevych02} for more background. A \emph{virtual endomorphism} of a group $A$ is a homomorphism from a finite index subgroup of $A$ to $A$. Let us call a virtual endomorphism $\phi$ of $A$ \emph{proper} if the intersection of the iterated preimages $\phi^{-n}(A)$ for all $n\in\N$ is trivial. For example, $\phi\colon 2\Z\to\Z$ sending $2m$ to $m$ is a proper virtual endomorphism, since $\phi^{-n}(\Z)=2^n\Z$, whereas the usual inclusion $2\Z\to\Z$ is not proper. Note that the $\phi^{-n}(A)$ all have finite index, so any group admitting a proper virtual endomorphism must be residually finite. Also note that the kernel of $\phi$ lies in every $\phi^{-n}(A)$, and so proper virtual endomorphisms are necessarily injective.

Virtual endomorphisms, even proper ones, need not be surjective, for example the virtual endomorphism $4\Z\to \Z$ sending $4m$ to $2m$ is proper but not surjective. One reason to focus on surjective virtual endomorphisms however, is that the finite indices $[\phi^{-n}(A)\colon \phi^{-(n+1)}(A)]$ are all the same \cite[Proposition~2.1]{nekrashevych02}.

\begin{definition}[Tree of cosets]
Let $\phi$ be a virtual endomorphism of a group $A$. Let $\tree_\phi$ be the tree with vertex set $\coprod_{n\ge 0} A/\phi^{-n}(A)$, i.e., with a vertex for each coset of a subgroup of the form $\phi^{-n}(A)$ (with $\phi^{-0}(A)$ interpreted to mean $A$, which is the root of the tree), and with an edge from $a\phi^{-n}(A)$ to $a\phi^{-(n+1)}(A)$ for each $n$ and each $a\in A$.
\end{definition}

The action of $A$ on the sets of cosets of each $\phi^{-n}(A)$ induces an action of $A$ by automorphisms on $\tree_\phi$. Note that if $\phi$ is surjective, so all the finite indices $[\phi^{-n}(A)\colon \phi^{-(n+1)}(A)]$ are the same, say they equal $d$, then $\tree_\phi\cong \tree_d$.

\medskip

We now introduce some new definitions, which will lead to the sorts of groups we care about acting nicely on certain trees of cosets.

\begin{definition}[$\Gamma$-stable, mutually stable]
Let $A$ be a group and $\Gamma\le \Aut(A)$. Let $\phi$ be a virtual endomorphism of $A$. We say that $\phi$ is \emph{$\Gamma$-stable} if $\gamma(\phi^{-n}(A))=\phi^{-n}(A)$ for all $\gamma\in\Gamma$ and all $n\in\N$. Now assume $\phi$ is bijective. Call $\phi$ and $\Gamma$ \emph{mutually stable} if $\phi$ is $\Gamma$-stable and $\phi \Gamma\phi^{-1}\le \Gamma$.
\end{definition}

Note that if $\phi$ is $\Gamma$-stable and $\phi$ is bijective then the expression $\phi \Gamma \phi^{-1}$ makes sense; more explicitly this is the set of all $\phi\circ \gamma|_{\phi^{-1}(A)}\circ \phi^{-1}$ for $\gamma\in\Gamma$. If $\phi$ is $\Gamma$-stable, then the action of $A$ on $\tree_\phi$ extends to an action of $A\rtimes\Gamma$ on $\tree_\phi$, via
\[
(a,\gamma).(a'\phi^{-n}(A)) = a\gamma(a')\phi^{-n}(A) \text{.}
\]

\begin{lemma}[Faithful]\label{lem:faithful}
Let $\phi$ be a surjective, proper virtual endomorphism of a group $A$. Let $\Gamma\le\Aut(A)$ such that $\phi$ is $\Gamma$-stable. Then the action of $A\rtimes \Gamma$ on $\tree_\phi$ is faithful.
\end{lemma}

\begin{proof}
Suppose $(a,\gamma)$ fixes every $a'\phi^{-n}(A)$, so $(a')^{-1}a\gamma(a')\in \phi^{-n}(A)$ for all $n\in\N$ and $a'\in A$. Since $\phi$ is proper, this implies that $(a')^{-1}a\gamma(a')=1$ for all $a'\in A$. The $a'=1$ case tells us that $a=1$, at which point the general $a'$ case tells us that $\gamma=\id$.
\end{proof}

Next we want conditions under which the action of $\Gamma$ on $\tree_\phi$ is self-similar. More precisely, when $\phi$ is surjective we have seen that $\tree_\phi\cong\tree_d$, for $d=[A\colon \phi^{-1}(A)]$, and we can get such an isomorphism by choosing some identification of the set of cosets of $\phi^{-1}(A)$ with the set $\{1,\dots,d\}$. Now call the action of $\Gamma$ on $\tree_\phi$ \emph{self-similar} if the induced action on $\tree_d$ is self-similar; this does not depend on the choose of identification with $\{1,\dots,d\}$.

\begin{lemma}[Self-similar]\label{lem:ss}
Let $\phi$ be a surjective, proper virtual endomorphism of a group $A$. Let $\Gamma\le\Aut(A)$ be such that $\phi$ and $\Gamma$ are mutually stable. Then the action of $A\rtimes\Gamma$ on $\tree_\phi$ is self-similar.
\end{lemma}

\begin{proof}
Rather than choosing an identification of the cosets of $\phi^{-1}(A)$ with $1,\dots,d$, which would become notationally cumbersome, let us equivalently just pin down how to identify the relevant subtrees of $\tree_\phi$ with $\tree_\phi$. Write $B$ for $\phi^{-1}(A)$, so $B$ is a finite index subgroup of $A$. Choose some transversal $T$ for $A/B$, so every coset of $A/B$ is represented by a unique element of $T$. Let $tB$ be a child of the root, for $t\in T$, and let $\tree_\phi(tB)$ be the subtree rooted at $tB$, i.e., the induced subgraph of $\tree_\phi$ whose vertex set consists of all the cosets contained in $tB$. Now we identify $\tree_\phi(tB)$ with $\tree_\phi$ via the isomorphism
\[
\delta_{tB} \colon \tree_\phi(tB) \to \tree_\phi
\]
sending $a\phi^{-(n+1)}(A)$ to $\phi(t^{-1}a)\phi^{-n}(A)$. Note that $\delta_{tB}^{-1}$ sends $a\phi^{-n}(A)$ to $t\phi^{-1}(a)\phi^{-(n+1)}(A)$.

Now let $(a,\gamma)\in A\rtimes \Gamma$, and we need to prove that the level-1 states of $(a,\gamma)$ lie in $A\rtimes \Gamma$. Let $tB$ be a child of the root, and let $t'B$ be the child of the root to which $(a,\gamma)$ sends $tB$, where here $t,t'\in T$. The level-1 state of $(a,\gamma)$ corresponding to $tB$ is the automorphism of $\tree_\phi$ given by
\[
\delta_{t'B} \circ (a,\gamma) \circ \delta_{tB}^{-1} \text{.}
\]
We compute that this sends the vertex $a'\phi^{-n}(A)$ to
\[
\phi\big((t')^{-1}a \gamma(t\phi^{-1}(a'))\big)\phi^{-n}(A)\text{,}
\]
which equals
\[
\phi\big((t')^{-1}a \gamma(t)\big) \gamma'(a') \phi^{-n}(A)\text{,}
\]
where $\gamma'\coloneqq \phi\circ\gamma\circ\phi^{-1}$. Since $\phi$ and $\Gamma$ are mutually stable, $\gamma'\in \Gamma$, so we are done.
\end{proof}

The culmination of all this is the following criterion for a group of a certain form to satisfy the Boone--Higman conjecture.

\begin{corollary}\label{cor:bh}
Let $A$ be a group with a surjective, proper virtual endomorphism $\phi$. Let $\Gamma\le\Aut(A)$ such that $\phi$ and $\Gamma$ are mutually stable. Then $A\rtimes\Gamma$ is a self-similar group. If it is finitely presented, then it, and hence all of its subgroups, satisfy the Boone--Higman conjecture.
\end{corollary}

\begin{proof}
By Lemmas~\ref{lem:faithful} and~\ref{lem:ss}, $A\rtimes\Gamma$ admits a faithful, self-similar action on the tree $\tree_\phi$, hence is a self-similar group. If it is finitely presented, then we conclude with Theorem~\ref{thrm:main}.
\end{proof}

\subsection{Examples}\label{ssec:examples}

First let us quickly recover the fact that $\Aff_n(\Z)=\Z^n\rtimes\GL_n(\Z)$ satisfies the Boone--Higman conjecture, via what essentially amounts to the same proof as in \cite{scott84} in different language. See also \cite{brunner98} and \cite[Lemma~6.3]{kionke23}.

\begin{example}\label{ex:GLnZ}
Let $A=\Z^n$ and let $\phi\colon (2\Z)^n\to\Z^n$ be $(2m_1,\dots,2m_n)\mapsto (m_1,\dots,m_n)$. This is a surjective, proper virtual endomorphism. Clearly $\phi$ and $\GL_n(\Z)$ are mutually stable, and $\Z^n\rtimes\GL_n(\Z)$ is finitely presented, so we conclude from Corollary~\ref{cor:bh} that $\Z^n\rtimes\GL_n(\Z)$ satisfies the Boone--Higman conjecture.
\end{example}

In fact we can generalize this quite a bit, as follows.

\begin{example}\label{ex:ring}
Let $R$ be a ring, viewed as an additive abelian group, and suppose $R$ has a principal left ideal $J=Rx$ of finite index. Assume $x$ is not a zero-divisor, so we have a group homomorphism $\varphi\colon J\to R$ given by $rx\mapsto r$. For $n\in\N$, let $A=R^{\oplus n}$ and let $\phi\colon J^{\oplus n}\to R^{\oplus n}$ be the direct sum of $n$ copies of $\varphi$. This is a surjective virtual endomorphism of $A$. Assume that $\bigcap_{k\in\N}Rx^k = \{0\}$; for example this holds if $R$ is a noetherian integral domain and $J$ is proper, thanks to the Krull intersection theorem. Since $\phi^{-k}(A)=(Rx^k)^{\oplus n}$, this means that the virtual endomorphism $\phi$ is proper. Note that in this case, $\phi$ is injective, and the inverse $\phi^{-1}\colon R^{\oplus n}\to J^{\oplus n}$ is given by scaling by $x$ on the right. Finally, let $\Gamma=\GL_n(R)$. Since $J$ is a left ideal, and since scaling by $x$ on the right commutes with multiplying by a matrix on the left, $\phi$ and $\GL_n(R)$ are mutually stable. Now we conclude from Corollary~\ref{cor:bh} that $R^{\oplus n} \rtimes \GL_n(R)$ is self-similar, and if it is also finitely presented then it satisfies the Boone--Higman conjecture.
\end{example}

As a concrete example, we can now prove Theorem~\ref{thrm:Q}, that every finitely generated subgroup of $\GL_n(\Q)$ satisfies the Boone--Higman conjecture.

\begin{proof}[Proof of Theorem~\ref{thrm:Q}]
Every finitely generated subgroup of $\GL_n(\Q)$ lies in $\GL_n(\Z[1/m])$ for some $m$, so thanks to Theorem~\ref{thrm:main} it is enough to prove that $\Z[1/m]^{\oplus n}\rtimes \GL_n(\Z[1/m])$ is finitely presented and self-similar. For self-similarity, we use Example~\ref{ex:ring}. Let $p$ be a prime not dividing $m$ and let $J=p\Z[1/m]$, so $J$ is a proper finite index principal ideal in $\Z[1/m]$. We therefore have a proper virtual endomorphism $\phi \colon J^{\oplus n} \to \Z[1/m]^{\oplus n}$, which is mutually stable with $\GL_n(\Z[1/m])$. Hence $\Z[1/m]^{\oplus n}\rtimes \GL_n(\Z[1/m])$ is self-similar. For finite presentability, we can prove this ``by hand'' as follows, as suggested to us by Francesco Fournier-Facio. First note that $\GL_n(\Z[1/m])$ is finitely presented \cite{borel76}, and $\Z[1/m]$ has presentation $\langle x_0,x_1,\dots\mid x_{i+1}^m=x_i$ for all $i\ge 0\rangle$. Thus $\Z[1/m]^{\oplus n}\rtimes \GL_n(\Z[1/m])$ has a presentation with a generator $x_i(j)$ for each $i\ge 0$ and $1\le j\le n$, plus finitely many other generators for the $\GL_n(\Z[1/m])$ factor, and defining relations as follows:
\begin{enumerate}
    \item $(x_{i+1}(j))^m=x_i(j)$ for all $i\ge 0$ and $1\le j\le n$,
    \item $x_i(j) x_k(\ell) = x_k(\ell) x_i(j)$ for all $i,k\ge 0$ and $1\le j,\ell\le n$ with $j\ne\ell$,
    \item $\gamma x_i(j) \gamma^{-1} = \gamma * x_i(j)$ for all generators $\gamma$ of $\GL_n(\Z[1/m])$ and all $i\ge 0$ and $1\le j\le n$, and
    \item finitely many other defining relations, for the $\GL_n(\Z[1/m])$ factor.
\end{enumerate}
Here $\gamma * x$ denotes the matrix $\gamma$ times the vector $x$. Now we want to reduce this to a finite presentation. For each $1\le j\le n$, let $\gamma_j\in \GL_n(\Z[1/m])$ be the diagonal matrix with $1/m$ in the $j$th diagonal entry and $1$ in the other diagonal entries, so $\gamma_j * x_i(j)=x_{i+1}(j)$, and $\gamma_j * x_i(\ell)=x_i(\ell)$ for $\ell\ne j$. Assume without loss of generality that the $\gamma_j$ are in our finite generating set for $\GL_n(\Z[1/m])$. Now for each $i>0$ and $1\le j\le n$, we can remove the generator $x_i(j)$ and replace all instances of it in the defining relations with $(\gamma_j)^i x_0(j) (\gamma_j)^{-i}$. At this point we have reduced the presentation to finitely many generators, and up to some straightforward manipulations the defining relations have become
\begin{enumerate}
    \item $\gamma_j x_0(j)^m \gamma_j^{-1} = x_0(j)$ for all $1\le j\le n$,
    \item each of $\gamma_j$ and $x_0(j)$ commutes with each of $\gamma_\ell$ and $x_0(\ell)$ for all $1\le j,\ell\le n$ with $j\ne\ell$,
    \item $\gamma x_0(j) \gamma^{-1} = \gamma*x_0(j)$ for all generators $\gamma$ of $\GL_n(\Z[1/m])$ and all $1\le j\le n$, and
    \item finitely many other defining relations.
\end{enumerate}
This is a finite set of relations, so we are done.
\end{proof}

\bibliographystyle{alpha}

\end{document}